\documentclass{article}

\usepackage{amsthm}
\usepackage{amsmath}
\usepackage{amssymb}
\usepackage{hyperref}
\usepackage{tikz-cd}
\usetikzlibrary{intersections}
\usepackage{mathrsfs}
\usepackage{array}
\usepackage{import}
\usepackage{pdfpages}
\usepackage[left=1.5in]{geometry}

\usepackage{titlesec}
\titlespacing*{\section}
{0pt}{5.5ex plus 1ex minus .2ex}{4.3ex plus .2ex}
\titlespacing*{\subsection}
{0pt}{5.5ex plus 1ex minus .2ex}{4.3ex plus .2ex}

\theoremstyle{definition}
\newtheorem{theorem}{Theorem}

\newtheorem{proposition}{Proposition}[section]
\newtheorem*{proposition*}{Proposition}
\newtheorem{defn}{Definition}[section]
\newtheorem{corollary}{Corollary}[section]
\newtheorem{conjecture}{Conjecture}

\newtheorem{lemma}{Lemma}[section]

\newcommand{\R}{\mathbb{R}}

\newcommand{\Address}{{
  \bigskip
  \footnotesize
  \textsc{Department of Mathematics and Statistics, University of Montreal, C.P.\, 6128 Succ. Centre-Ville, Montreal, QC, H3C 3J7, Canada}\par\nopagebreak
  \textit{E-mail address:} \texttt{filip.brocic@umontreal.ca}\hspace{5mm}  
}}

\title{\vspace{-2.0cm}A note on the capacities of Lagrangian $p$-sum}
\author{Filip Bro\'ci\'c}

\begin{document}

\maketitle

\begin{abstract}
In this short note, we construct an explicit embedding of the rescaling of the $p$-sum $K\oplus_p K^{\circ}$ of the centrally symmetric convex domain $K$ and its polar $K^{\circ}$ to the product $K \times K^{\circ}$. The rescaling constant is sharp in some cases. Additionally, we comment on the strong Viterbo conjecture for $K\oplus_p K^{\circ}$.
\end{abstract}

\section{Introduction}

In \cite{Br23}, the author constructed a symplectic embedding $e:B^{2n}(4)\to K \times K^{\circ}$ from the ball of capacity $4$ to the product $K \times K^{\circ}$, in the case when $K$ is the unit ball of an Euclidian norm, and $K^{\circ} = \{y\in \R^n \mid \langle y, x \rangle, \forall x \in K \}$ is the polar set. In Proposition \ref{embedding p sum} we generalize this construction to a larger class of domains.  Our domains are constructed using an arbirtrary norm $\| \cdot \|:\mathbb{R}^n \to [0, +\infty)$ on $\mathbb{R}^n$. Let $\| \cdot \|_*$ be the dual norm to $\| \cdot \|$, considered also on $\mathbb{R}^n$. For $p\in[1, \infty)$ set $\|(x,y)\|_{p,*}:= (\|x \|^p + \| y \|_*^p)^{1/p}$ and denote
$$
K\oplus_p K^{\circ} := \{ (x,y) \in \R^{2n} \mid \|x \|^p + \| y \|_*^p <1\},
$$
the unit ball in the norm $\|(x,y)\|_{p,*}$, where $K:= \{x \in \R^n \mid \| x \| < 1\}$ is the unit ball in the norm $\| \cdot \|$, and $K^{\circ}$ is its polar set, which coincides with the unit ball in the norm $\| \cdot \|_*$. For $p=\infty$ we have $\|(x,y)\|_{\infty,*} = \max\{\| x\|, \| y \|_*\}$, i.e., $K\oplus_{\infty} K^{\circ} = K \times K^{\circ}$.  In Proposition \ref{embedding p sum}, we construct a symplectic embedding from a rescaling of $K\oplus_p K^{\circ}$ to $K \times K^{\circ}$. Our construction immediately implies estimates of the \textit{symplectic} capacities of $K\oplus_p K^{\circ}$. 

\begin{defn}
Symplectic capacity is a map $c: \mathcal{O}(\R^{2n}) \to [0, +\infty]$ from the set of open subsets of $\R^{2n}$ to non-negative real numbers such that:
\begin{itemize}
\item[1.)] (Monotonicity) If there is a symplectic embedding $\psi: \mathcal{V} \to \mathcal{U}$ then $c(\mathcal{V}) \leq c(\mathcal{U})$;
\item[2.)] (Rescaling) If $\lambda>0$ then $c(\lambda \mathcal{U}) = \lambda^2 c(\mathcal{U})$;
\item[3.)] (Non-triviality) $0<c(B^{2n}(1)) \leq c(Z^{2n}(1)) < \infty$,
\end{itemize}
where $B^{2n}(R):=\{z\in \R^{2n} \mid \pi \|z\|^2 \leq R\}$, and $Z^{2n}(R)=B^{2}(R) \times \R^{2n-2}$ is the symplectic cylinder. Additionally, capacity is called \textit{normalized} if it satisfies $c(B^{2n}(1)) = c(Z^{2n}(1)) = 1$.
\end{defn}

The $p$-sum of 2-dimensional Euclidian Lagrangian discs, and their Gromov width, was studied in \cite{OR22}. Symplectic capacities of the symplectic $p$-sum were considered in \cite{KO23}. Here, we prove the following theorem for all centrally symmetric convex domains $K \subset \R^n$.

\begin{theorem}\label{main thm}
For every symplectic capacity $c$
$$
c(K\oplus_p K^{\circ}) \leq \frac{\Gamma^{2}(1+\frac{1}{p})}{\Gamma (1+\frac{2}{p})} c(K\times K^{\circ}).
$$
\end{theorem}
\begin{proof}
This follows from Proposition \ref{embedding p sum}, and the rescaling axiom for symplectic capacities.
\end{proof}

Main examples of normalized capacities are \textit{Gromov width} $$Gr(\mathcal{U}) = \sup\{ R \mid e: B^{2n}(R) \to \mathcal{U}, \hspace{1mm} e^*{\omega_{st}} = \omega_{st} \},$$ where $\omega_{st} = \sum dx_i \wedge dy_i$, and cylindircal capacity $$c_{Z}(\mathcal{U}) = \inf \{R \mid e: \mathcal{U} \to Z^{2n}(R), \hspace{1mm} e^*{\omega_{st}} = \omega_{st}\}.$$ The fact that these are (normalized) symplectic capacities is a reformulation of the famous Gromov's non-squeezing theorem (see \cite{Gr85}). Other examples of normalized symplectic capacities are Ekhlan-Hofer capacity and Hofer-Zehnder capacity (see \cite{EH89, HZ90}). They coincide on convex sets, and for convex set $X \subset R^{2n}$ they are equal to the minimal period of the Reeb orbit on $\partial X$, and we denote this quantity $c_{EHZ}(X)$. In \cite{Ki19}, the capacity $c_{EHZ}$ is calculated for convex polytopes. It was conjectured in \cite[\S5]{Vi00} that for any normalized symplectic capacity, if $X$ is a convex set then 
$$
\frac{c^{n}(X)}{n!}\leq \mathrm{\mathrm{Vol}}(X).
$$
This conjecture would follow from what is known as the strong Viterbo conjecture:
\begin{conjecture}[Strong Viterbo conjecture]
All normalized capacities coincide on convex sets.
\end{conjecture} 
It is clear from the definition of a symplectic capacity, that every normalized capacity $c$ is between Gromov width and cylindrical capacity, i.e., $Gr(\mathcal{U}) \leq c(\mathcal{U}) \leq c_{Z}(\mathcal{U})$. Hence, proving strong Viterbo conjecture is equivalent to proving that $Gr(X) = c_{Z}(X)$ for every convex set $X$.

By projecting to a smart choice of a symplectic plane, it follows from \cite[Remark 4.2]{AAKO14} that the cylindrical capacity of $K \times K^{\circ}$ satisfies $c_Z (K \times K^{\circ}) \leq 4$. The main result of \cite{AAKO14} shows\footnote{By proving this, they were able to relate Viterbo conjecture with Mahler's conjecture from Convex Geometry, which states that $\mathrm{Vol}(K \times K^{\circ})\geq 4^n / n!$.} that $c_{EHZ}(K \times K^{\circ})= 4$. Consequently, since every normalized capacity is bounded by the cylindrical capacity, we have $c_Z (K \times K^{\circ}) = 4$. From Theorem \ref{main thm} we get $$
c_{Z}(K\oplus_p K^{\circ})  \leq 4 \frac{\Gamma^{2}(1+\frac{1}{p})}{\Gamma (1+\frac{2}{p})}.$$
Ostrover and Ramos proved in \cite{OR22} that the $p$-sum of 2-dimensional Euclidian Lagrangian discs is a toric domain. They calculated the Gromov width of $K\oplus_p K^{\circ}$, for $K = \{(x_1, x_2) \in \R^2\mid x_1^2 + x_2^2 \leq 1\}$. In the case $p \geq 2$ they showed:
$$
Gr(K\oplus_p K^{\circ}) = 4 \frac{\Gamma^{2}(1+\frac{1}{p})}{\Gamma (1+\frac{2}{p})}.
$$
As a corollary, we get:

\begin{corollary}
If $p\geq 2$, the strong Viterbo conjecture holds for $$B^2(1) \oplus_{p} B^2(1)=\{(x,y) \in \R^4 \mid \|x\|_{st}^p + \|y\|_{st}^p < 1\}.$$
\end{corollary}

The same result follows from \cite{GHR22}, where they proved that the strong Viterbo conjecture holds for \textit{monotone} toric domains in $\R^4$. This result was generalized to arbitrary dimension in \cite{CGH23}. However, for centrally symmetric and convex $K$ it is not known if the $p$-sym $K\oplus_p K^{\circ}$ is symplectomorphic to a toric domain in higher dimensions. 

Our construction also imples a new proof of \cite[Corollary 5.4]{BK22} for the particular case when $X=K\oplus_{2} K^{\circ}$.

\begin{theorem}\label{EHZ capacity thm}
$$
2 + \frac{1}{n}\leq c_{EHZ}(K\oplus_{2}K^{\circ})\leq \pi.
$$
\end{theorem}

\begin{proof}
Using \cite[Theorem 1.1]{AK17}, we have for a centrally symmetric convex set $T$
$$
\frac{2+ 1/n}{\| J\|_{T^{\circ} \to T}} \leq c_{EHZ}(T),
$$
where $\| J\|_{T^{\circ} \to T} = \sup_{u,v \in T^{\circ}} \langle J u , v \rangle$ (see also \cite[Theorem 1.6]{GO16}). For $T = K\oplus_{2}K^{\circ}$ using Lemma \ref{dual of 2-sum} we have $J T^{\circ} = T$, which further implies
\begin{equation} \label{normJ}
\sup_{u,v \in T^{\circ}} \langle J u , v \rangle = \sup\left\{ \lambda \langle u_x, u_x\rangle + \mu \langle u_y, u_y\rangle  \mid  \begin{aligned}\|u_x\|^2 + \|u_y\|_{*}^2 &\leq 1,\\\lambda^2 \|u_x\|^2 + \mu^2 \|u_y\|_{*}^2 &\leq 1\end{aligned} \right\}.
\end{equation}
Where $(u_x, u_y)$ is a coordinate representation of $u\in \R^{2n} = \R^{n} \times \R^{n}$. The equality in equation (\ref{normJ}) follows from analyzing the critical points of the function $$F(u,v, \lambda, \mu)= \langle J u, v \rangle - \frac{\lambda}{2}(\|u\|^2_{2,*} -1) - \frac{\mu}{2}(\|Jv\|^2_{2,*} -1).$$   Also, we have
$$
 2 (\lambda \|u_x\| \|u_x\|_{*} + \mu \|u_y\| \|u_y\|_{*}) \leq \lambda^2 \|u_x\|^2 + \|u_x\|_{*}^2 + \mu^2 \|u_y\|_{*}^2 + \|u_y\|^2 \leq 2,
$$
which implies $\| J\|_{T^{\circ} \to T} \leq 1$. On the other hand, for $\|u_y\|=1$, setting $u=(0, u_y)$ and $v=(\nabla N(u_y), 0)$ we get
$$
\langle Jv, u \rangle = \langle \nabla N(u_y), u_y \rangle = \|u_y\| = 1,
$$
since both $u, v \in T^{\circ}$ (see Lemma \ref{DualNorm}), we get $\| J\|_{T^{\circ} \to T}=1$.

For the other inequality, we use the symplectic embedding from Proposition \ref{FromEuclidian}, together with the result $c_{EHZ}  (K\times K^{\circ}) = 4$ from \cite{AAKO14}:

$$
\frac{4}{\pi} c_{EHZ}(K\oplus_{2}K^{\circ}) = c_{EHZ}\left( \frac{2}{\sqrt{\pi} } K\oplus_{2}K^{\circ} \right) \leq  c_{EHZ}  (K\times K^{\circ}) = 4.
$$
\end{proof}

\subsection*{Acknowledgements}
I am very thankful to my Ph.D. advisors Octav Cornea and Egor Shelukhin for their constant support during my studies. This work benefited a lot from discussions with Dylan Cant, Pazit Haim-Kislev, and Yaron Ostrover. I would also like to thank Vinicius Ramos and Dan Cristofaro-Gardiner for their comments on the earlier version of this document. This research is partially supported by \'Etudes sup\'erieures et postdoctorales (ESP) scholarship and Fondation Courtois.

\section{Proofs}

We will work with norms that are smooth (or at least $C^1$) away from the origin. Set $N(x):= \| x \|$, and denote $\nabla N(x)$ the gradient with respect to the standard inner product.

\begin{proposition}\label{FromEuclidian}
There exists a relative symplectic embedding $$e:\sqrt{4/\pi} (K\oplus_{2}K^{\circ}) \to K\times K^{\circ}.$$
\end{proposition}

\begin{proof}
Let  $e:B^2(4) \to (-1, 1) \times (-1, 1)$ be a symplectic embedding from \cite[Lemma 3.7.]{Br23}. It is of the form $e(x,y) = \left(f(x), \frac{1}{f'(x)}y\right)$, where $$f(x)=\frac{2}{\pi}\mathrm{arcsin}\left(\sqrt{\frac{\pi}{4}}x\right) + \frac{x}{2} \sqrt{\frac{4}{\pi} - x^2}.$$
For higher dimensions we set $\varphi(x):= \frac{f(\|x\|)}{\|x\|}x$. Since $f$ is odd and analytic $\varphi$ is smooth. It is easy to see that, for $x \neq 0$
$$
D\varphi(x)h = \left(f'(\|x\|)-\frac{f(\|x\|)}{\|x\|}\right) \langle \nabla N(x) , h \rangle \frac{x}{\|x\|} + \frac{f(\|x\|)}{\|x\|}h,
$$
and
\begin{align}
     \|D\varphi(x)h\| &\geq  \frac{f(\|x\|)}{\|x\|} \|h\| - \left\vert f'(\|x\|)-\frac{f(\|x\|)}{\|x\|} \right\vert |\langle \nabla N(x) , h \rangle|  \nonumber\\
     & \geq   \frac{f(\|x\|)}{\|x\|} \|h\|  - \left(\frac{f(\|x\|)}{\|x\|} - f'(\|x\|)\right) \| \nabla N(x) \|_* \| h \| \nonumber \\
     &\geq f'(\|x\|) \|h\|. \label{LowerBoundonDFi}
\end{align}
   The second inequality follows from $\frac{f(t)}{t} > f'(t)$ for $t>0$ and $|\langle \nabla N(x) , h \rangle| \leq \| \nabla N(x) \|_* \| h \|$, and the last inequality follows from Lemma \ref{DualNorm} where we show that $\| \nabla N(x) \|_* = 1$ . Since $f'(t)>0$ for $t\in \left[0, 2 \sqrt{\frac{ab}{\pi}}\right)$ we get that $D\varphi(x)$ is invertible. Now define symplectic embedding $e:\sqrt{4/\pi} (K\oplus_{2}K^{\circ}) \to K\times K^{\circ}$ as $$e(x,y):=\left(\varphi(x), (D\varphi(x)^{-1})^* y\right).$$ Since $\| D \varphi (x) \|_{op} = \| D \varphi (x)^* \|_{op} \geq f'(\| x \|)$, we have $\| (D \varphi (x)^{-1})^* \|_{op} \leq 1/ f'(\|x\|)$, which implies that $Im(e) \subset K\times K^{\circ}$.
 To see that such embedding $e$ is symplectic it is enough to note that $e^* \lambda_{st} = \lambda_{st}$ where $\lambda_{st} = \sum y_i dx_i$.
\end{proof}

More generally, take two continuous functions $g:[0,a] \to [0, +\infty)$ and $h:[0,b] \to [0, +\infty)$, which are positive for $x\in[0,a)$ (respectively $x\in [0,b)$) and $\int_0^a g(t)dt = \int_0^b h(t) dt.$ Define a function $f:[0,a] \to [0,b]$ implicitly by
$$
\int_0^x g(t)dt = \int_0^{f(x)} h(t) dt.
$$
If $f''\leq 0$, the diffeomorphism $\varphi: B^n(a) \to B^n(b)$ defined by $\varphi(x):= (f(\|x\|)/\|x\|) x$ induces a symplectic embedding from the set $B^{2n}_{g}:= \{(x,y) \in \mathbb{R}^{2n} \mid \|x\| \leq a, \| y \|_* \leq g(\| x \|) \}$ to $B^{2n}_{h}:= \{(x,y) \in \mathbb{R}^{2n} \mid \|x\| \leq b, \| y \|_* \leq h(\| x \|) \}$. 

Taking $b=1$, $h=1$ and $g(t):=\left(\frac{\Gamma (1+\frac{2}{p})}{\Gamma^2 \left(1+\frac{1}{p}\right)} - t^p\right)^{1/p}$ we get:

\begin{proposition}\label{embedding p sum}
There exists a symplectic embedding $e: \frac{\sqrt{\Gamma (1+\frac{2}{p})}}{\Gamma \left(1+\frac{1}{p}\right)} (K\oplus_p K^{\circ}) \to K\times K^{\circ}.$

\end{proposition}

Note that the construction of an embedding works only for $K$ with (at least) $C^1$ boundary. However, by approximating the non-smooth domains with smooth ones the conclusion of the Theorem \ref{main thm} still holds. The embedding $e$ is of class $C^1$, if the boundary of $K$ is smooth, one can take a smoothing of $\varphi$ near the origin and construct a smooth symplectic embedding $e: \frac{\sqrt{\Gamma (1+\frac{2}{p})}}{\Gamma \left(1+\frac{1}{p}\right)} (K\oplus_p K^{\circ}) \to (1+\epsilon)K\times K^{\circ}$, for every $\epsilon>0$. 

The embeddings from Proposition \ref{FromEuclidian} are not filling the volume, which we can see from the following Lemma (see \cite[Lemma 2.2.]{BK22}).

\begin{lemma}\label{VolumeOfEuclidian}
$$\mathrm{Vol}(K\oplus_{p}K^{\circ}) = \frac{\Gamma^2(\frac{n}{p} + 1)}{\Gamma(\frac{2n}{p}+1)}  \mathrm{Vol}(K\times K^{\circ}).$$
\end{lemma}
\begin{proof}
\begin{align}
\mathrm{Vol}(K\oplus_{p}K^{\circ})&=\int_{\| x \| \leq 1} \int_{\| y \|_* \leq (1 - \|x\|^p)^{1/p}}dy dx = \int_{\| x\| \leq 1} \mathrm{Vol}((1 - \|x\|^p)^{1/p} K^{\circ} ) dx  \nonumber\\ 
 &= \mathrm{Vol}(K^{\circ}) \int_{\|x \|\leq 1} (1 - \|x\|^p)^{n/p} dx. \label{IntegrationOfVolume}
\end{align}
Here, we have use that $\mathrm{Vol}(r K)=r^n \mathrm{Vol}(K)$.  By Fubini's Theorem, we have 
\begin{align}
\int_{\|x \|\leq 1} (1 - \|x\|^p)^{n/p} dx&= \int_{\|x \|\leq 1} \int_0^{(1 - \|x\|^p)^{n/p}} dt dx = \int_0^1 \int_{\| x\| \leq (1-t^{p/n})^{1/p}} dx dt\nonumber \\
 &= \mathrm{Vol}(K) \int_0^1\left(1-t^{\frac{p}{n}}\right)^{n/p} dt, \label{Fubini}
\end{align}
By setting $t=u^{n/p}$ we get
\begin{align}
\int_0^1\left(1-t^{\frac{p}{n}}\right)^{\frac{n}{2}} dt &= \frac{n}{p} \int_0^1 (1-u)^{\frac{n}{p}} u^{\frac{n}{p}-1} du = \frac{n}{p} B\left(\frac{n}{p}+1, \frac{n}{p}\right) \nonumber\\
&= \frac{n}{p} \frac{\Gamma(\frac{n}{p} + 1) \Gamma(\frac{n}{p})}{\Gamma(\frac{2n}{p}+1)} = \frac{\Gamma^2(\frac{n}{p} + 1)}{\Gamma(\frac{2n}{p}+1)}, \label{GammaFunction}
\end{align}
where $B(a,b)=\int_0^1(1-t)^{a-1} t^{b-1} dt$ is the Beta function, and $\Gamma(a)=\int_0^{\infty}e^t t^{a-1} dt$ is the Gamma function.
Combining (\ref{IntegrationOfVolume}), (\ref{Fubini}) and (\ref{GammaFunction})  we get
$$
\mathrm{Vol}(r K\oplus_{p}K^{\circ}) =  \frac{\Gamma^2(\frac{n}{p} + 1)}{\Gamma(\frac{2n}{p} + 1)} \mathrm{Vol}(K\times K^{\circ})  r^{2n} .
$$

\end{proof}

The following result is standard in Convex Geometry (see e.g. \cite[Remark 1.7.14]{Sch13}). We include the proof for the sake of completeness. 
\begin{lemma}\label{DualNorm}
The dual norm $\| \nabla N(x) \|_* $ of the gradient $\nabla N(x)$ is equal to 1.
\end{lemma}
\begin{proof}
It is east to show that $\| \nabla N(x) \|_* \geq 1$. Consider a derivative with respect to $s$ of the expression $N(sx)=sN(x), s>0$ we get $\langle \nabla N(x), x \rangle = N(x) = \| x \|.$ Now, from the definition of the dual norm, we have $\langle \nabla N(x), x \rangle \leq \| x \| \| \nabla N(x) \|_*,$ which is an analog of the Cauchy-Schwartz inequality, hence
$$
\| \nabla N(x) \|_* \geq \frac{\langle \nabla N(x), x \rangle}{\| x \|} = \frac{\| x \|}{\| x \|}=1.
$$
We will calculate $\| \nabla N(x) \|_* = \sup_{\| y \| =1} \langle \nabla N(x), y \rangle$ using Lagrange multipliers. Consider a function $F(y,\lambda) = \langle y, \nabla N(x) \rangle - \lambda (N(y) - 1)$. The derivatives are
\begin{align*}
D_y F(y,\lambda) h & = \langle h, \nabla N(x) \rangle - \lambda \langle h, \nabla N(y) \rangle = \langle h, \nabla N(x) - \lambda \nabla N(y) \rangle, \\
D_\lambda F(y, \lambda) &= N(y) - 1. 
 \end{align*}
 Since $\nabla N(sy) = \nabla N(y), s>0$, and $\nabla N(-y) = - \nabla N(y)$, critical points of $F$ contain the set $(y, \lambda) \in \left\{ \left( \frac{x}{\| x \|}, 1 \right), \left( -\frac{x}{\| x \|}, 1 \right), \left( \frac{x}{\| x \|}, -1 \right), \left( -\frac{x}{\| x \|}, -1 \right)\right\}.$ The other critical points are of form $(\frac{y}{\| y \|}, \lambda)$, where $\nabla N(x) = \lambda \nabla N(y)$. For such $y$ and if $\lambda > 0$,  we will show in Lemma \ref{ConvexFunction} that $$\left\langle \frac{y}{\| y \|}, \nabla N(x) \right\rangle = \left\langle \frac{x}{\| x\|}, \nabla N(x) \right\rangle.$$  
 Hence the maximum is achieved at $\left( \frac{x}{\| x \|}, 1 \right)$ which implies $ \| \nabla N(x) \|_* = \left\langle \frac{x}{\| x\|}, \nabla N(x) \right\rangle = 1.$
\end{proof}

\begin{lemma}\label{ConvexFunction}
If $\nabla N(x) = \lambda \nabla N(y)$ for $\lambda>0$ then $\left\langle \frac{y}{\| y \|}, \nabla N(x) \right\rangle = \left\langle \frac{x}{\| x\|}, \nabla N(x) \right\rangle.$
\end{lemma}

\begin{proof}
Since $\| \cdot \|$ is a norm, the unit ball $B^n(1)= \{ x \in \mathbb{R}^n \mid \| x \| \leq 1 \}$ is a convex set. Now, we will consider $A:=B^{n}(1) \cap Span\{ x, y \}$. This is again a convex set, with the property that the normal vector to the $\partial A$ at $\frac{x}{\| x \|}$ is proportional to the normal vector at $\frac{y}{\| y \|}$. Since $A$ is convex, and $\partial A$ is smooth, we have that the segment $\left[ \frac{x}{\| x \|}, \frac{y}{\| y\|} \right] \subset \partial A$ is contained in the boundary (see Figure \ref{Domain}). This further implies that $\langle \frac{y}{\| y \|}, \nabla N(x) \rangle = \langle \frac{x}{\| x\|}, \nabla N(x) \rangle$.
\end{proof}

\begin{figure}[h!]
 \centering
   \begin{tikzpicture}
\draw[color=blue!40!cyan, very thick] (-2,2)  to (-1, 2) to [out=0, in=180] (0,2.5) to [out=0, in=180] (0.5, 2) to  (1,2);
\node[color=blue!40!cyan] at (-0.25,1.75) {$\| x \| = 1$};
\draw[->] (0,1) to (-2,2); 
\node[left] at (-1,1.2) {$\frac{x}{\| x\|}$};
\draw[->] (0,1) to (1,2);
\node[right] at (0.5,1.2) {$\frac{y}{\| y \|}$};
\draw[->] (-2,2) to (-2,2.5);
\node[left] at (-2,2.5){$\nabla N(x)$};
\draw[->] (1,2) to (1,2.5);
\node[right] at (1,2.5) {$\nabla N(y)$}; 
\end{tikzpicture}
    \caption{(Non)-convex domain}
    \label{Domain}
\end{figure}
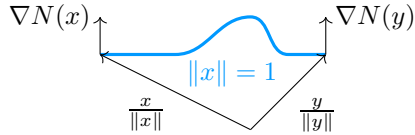

\begin{lemma}\label{dual of 2-sum}
If  a convex body $K$ is the unit ball of the norm $\sqrt{\|x\|^2 + \|y\|_{*}^2}$ then the dual body $K^{\circ}$ is given by the norm $\sqrt{\|x\|_{*}^2 + \|y\|^2}$. 
\end{lemma}
\begin{proof}
Dual of the norm $\|(x,y)\|_{2,*} = \sqrt{\|x\|^2 + \|y\|_{*}^2}$ is equal to $$\|(x,y)\|_{2,*}^* = \sup_{\|a\|^2 + \|b\|_{*}^2=1} \langle a,x \rangle + \langle b, y \rangle.$$
Applying Lagrange multiplier theorem to the function $$F(a,b, \lambda) =  \langle a,x \rangle + \langle b, y \rangle - \frac{\lambda}{2}( \|a\|^2 + \|b\|_{*}^2 - 1),$$
we get that the critical points satisfy $\lambda \|a\| \nabla N(a) = x$, and $\lambda \|b\|_{*} \nabla N^*(b) = x$. Using that $\| \nabla N(a)\|_{*} = 1$ and $\| \nabla N^*(b) \| = 1$ we get $\lambda = \sqrt{\|x\|_{*}^2 + \|y\|^2 }$, which further implies
$$
\|(x,y)\|_{2,*}^* = \langle a, \lambda \|a\| \nabla N(a)\rangle + \langle b, \lambda \|b\|_{*} \nabla N^*(b) \rangle = \sqrt{\|x\|_{*}^2 + \|y\|^2 }.
$$
\end{proof}
{\footnotesize
\bibliography{citations}
\bibliographystyle{alpha}
}
\Address
\end{document}